\documentclass[11pt]{elsarticle}
\usepackage{mathrsfs}

 \usepackage{graphics}
\usepackage{amssymb}
\usepackage{amsfonts}
\usepackage{pifont}
\usepackage{amsmath}
\usepackage{latexsym}
\usepackage{bm}
\usepackage{amssymb}
\usepackage{amsthm}
\usepackage{graphics}
\usepackage{epstopdf}
\usepackage{tikz}
\def\bi{\begin{align}}
\def\bin{\begin{align*}}
\numberwithin{equation}{section} \numberwithin{figure}{section}
\numberwithin{table}{section}
\newtheorem{thm}{Theorem}[section]
\newtheorem{cor}[thm]{Corollary}

\newtheorem{defn}[thm]{Definition}
\newtheorem{rem}[thm]{Remark}
\newtheorem{exa}{Example}[section]

\newcommand{\si}{\sigma}

\newcommand{\dif}{\mathrm{d}}

\newcommand{\ba}{\begin{array}}
\newcommand{\ea}{\end{array}}

\begin{document}

\begin{frontmatter}



\title{Symmetric integrators based on continuous-stage Runge-Kutta-Nystr\"{o}m methods for reversible
systems}

%\author{}
\author[a,b]{Wensheng Tang}
\ead{tangws@lsec.cc.ac.cn}
\address[a]{College of Mathematics and Statistics,\\
    Changsha University of Science and Technology,  Changsha 410114, China}
\address[b]{Hunan Provincial Key Laboratory of \\
    Mathematical Modeling and Analysis in Engineering,  Changsha 410114, China}
\author[c]{Jingjing Zhang\corref{cor1}}
\ead{jjzhang06@outlook.com}\cortext[cor1]{Corresponding author.}
\address[c]{School of Science, East China Jiaotong University, Nanchang 330013, China}
\author[]{}
\author[]{}
%\ead{}

\begin{abstract}

In this paper, we study symmetric integrators for solving
second-order ordinary differential equations on the basis of the
notion of continuous-stage Runge-Kutta-Nystr\"{o}m methods. The
construction of such methods heavily relies on the Legendre
expansion technique in conjunction with the symmetric conditions and
simplifying assumptions for order conditions. New families of
symmetric integrators as illustrative examples are presented. For
comparing the numerical behaviors of the presented methods, some
numerical experiments are also reported.

\end{abstract}


\end{frontmatter}

% \linenumbers

\section{Introduction}
\label{}



Numerical integration that preserves at least one of geometric
properties of a given dynamical system has attracted much attention
in these years \cite{Fengq10sga,hairerlw06gni,sanzc94nhp}. As
suggested by Kang Feng \cite{Feng84ods,Fengq10sga}, it is natural to
look forward to those discrete systems which preserve as much as
possible the intrinsic properties of the continuous system --- this
is a truly ingenious idea for devising ``good" integrators to
properly simulate the evolution of various dynamical systems with
geometric features. It is evidenced that numerical methods with such
a special purpose can not only perform a more accurate long-time
integration than those traditional methods without any
geometric-feature preservation, but also produce an improved
qualitative behavior \cite{hairerlw06gni}. Such type of methods,
generally associated with the terminology ``geometric integration",
are distinguished by the geometric properties they inherit,
including symplectic methods for Hamiltonian systems, symmetric
methods for reversible systems, volume-preserving methods for
divergence-free systems, invariant-preserving methods for
conservative systems, multi-symplectic methods for Hamiltonian
partial differential equations etc. For more details, we refer the
interested readers to
\cite{Feng84ods,Fengq10sga,hairerlw06gni,sanzc94nhp,Hong05aso} and
references therein.

Reversible systems and reversible maps are of interest in both
aspects of theoretical study and numerical simulation for many
differential equations \cite{hairerlw06gni}. Let $\rho$ be an
invertible linear transformation in the phase space of a first-order
system given by\footnote{If the system is non-autonomous, we can
introduce an extra equation namely $\dot{t}=0$ to rewrite the
original system as an autonomous system.} $z'=f(z)$, then the system
is called $\rho$-reversible if \cite{hairerlw06gni}
$$\rho f(z)=-f(\rho z),\,\text{for}\;\forall\, z,$$
and a map $\phi(z)$ is called $\rho$-reversible if
$$\rho\circ\phi=\phi^{-1}\circ\rho.$$
Particularly, it is shown in \cite{hairerlw06gni} that all
second-order systems with the form $z''=f(z)$ are reversible as they
can be transformed into reversible first-order systems. In addition,
notice that the exact flow of a reversible system is a reversible
map, it is therefore natural to find a numerical method $\Phi_h$,
which is better referred to as a reversibility-preserving
integrator, such that it is also a reversible map (i.e.,
$\rho\circ\Phi_h=\Phi_h^{-1}\circ\rho$). It is known that a number
of symmetric integrators automatically possess this property, e.g.,
all symmetric Runge-Kutta (RK) methods, some partitioned Runge-Kutta
(PRK) methods for special partitioned systems, some composition and
splitting methods, and standard projection methods for differential
equations on special manifolds (see \cite{hairerlw06gni}, page 145).
To be specific, we quote the following result from
\cite{Stofferz95vsf}.
\begin{thm}\cite{Stofferz95vsf}\label{rever-preser}
A Runge-Kutta  method or a Runge-Kutta-Nystr\"{o}m (RKN) method is
reversible iff it is symmetric.
\end{thm}
Thanks to the property of reversibility preservation, symmetric
integrators often have an excellent long-time numerical behavior
than those non-symmetric integrators for reversible systems
\cite{hairerlw06gni}. So far, a wide variety of effective symmetric
integrators have been proposed (see
\cite{burntons98grk,Cash05avs,Chan90osr,Fengq10sga,hairerlw06gni,
Mclachlanqt98nit,Stofferz95vsf} and references therein).

In the context of geometric integration, the greatest interest has
been given to the development of symplectic integrators for solving
Hamiltonian systems over the last decades
\cite{Fengq10sga,hairerlw06gni,sanzc94nhp}. However, if the
Hamiltonian $H(p,q)$ satisfies $H(-p,q)=H(p,q)$, then the system is
reversible with respect to the linear transformation $\rho: (p,
q)\mapsto (-p, q)$. Particularly, a well-known class of separable
Hamiltonian systems determined by the Hamiltonian
$H(p,q)=\frac{1}{2}p^TMp+U(q)$ happens to be such type of reversible
systems. Therefore, it makes sense for devising a numerical method
that preserves symplecticity and reversibility at the same time, and
fortunately, this has been shown to be an attainable goal (see
\cite{burntons98grk,Fengq10sga,hairernw93sod,hairerlw06gni,okunbors92ecm,
Tangs12ana,Tangs14cor} and references therein). Besides, a numerical
method which is energy-preserving and reversibility-preserving can
also be of interest
\cite{brugnanoit10hbv,brugnanoit15aoh,brugnanoi16lmf,brugnanoit12asf,brugnanoi18lis,hairer10epv,
quispelm08anc,Tangs12ana,Tangs14cor}.

In recent years, numerical methods with infinitely many stages
including continuous-stage Runge-Kutta (csRK) methods,
continuous-stage partitioned Runge-Kutta (csPRK) methods and
continuous-stage Runge-Kutta-Nystr\"{o}m (csRKN) methods are
presented and discussed by several authors, see
\cite{brugnanoit10hbv,hairer10epv,liw16ffe,miyatake14aee,miyatake16aco,Tangs12ana,
Tangs12tfe,Tangs14cor,Tanglx15cos,Tangz15spc,Tangsz15hos,Tang18ano,
Tang18csr,Tang18aef}. They can be viewed as the natural
generalizations of numerical methods with finite stages (e.g.,
classical RK methods). It is shown in
\cite{Tangs14cor,Tanglx15cos,Tangz15spc,Tangsz15hos,Tang18ano,
Tang18csr,Tang18aef} that by using continuous-stage methods many
classical RK, PRK and RKN methods of arbitrary order can be derived,
without resort to solving the tedious nonlinear algebraic equations
(associated with order conditions) in terms of many unknown
coefficients. The construction of continuous-stage methods seems
much easier than that of those traditional methods with finite
stages, as the associated Butcher coefficients are ``continuous" or
``smooth" functions and hence they can be treated by using some
analytical tools
\cite{Tangs14cor,Tanglx15cos,Tangz15spc,Tangsz15hos,Tang18ano,
Tang18csr,Tang18aef}. Moreover, as presented in
\cite{brugnanoit12asf,hairer10epv,liw16ffe,miyatake14aee,miyatake16aco,Tangs14cor,
Tanglx15cos,Tangz15spc,Tangsz15hos,Tang18ano,Tang18csr,Tang18aef},
numerical methods serving some special purpose including
symplecticity-preserving methods for Hamiltonian systems, symmetric
methods for reversible systems, energy-preserving methods for
conservative systems can also be established within this new
framework. Besides, a well known negative result we have to mention
here is that no RK methods is energy-preserving for general
non-polynomial Hamiltonian systems \cite{Celledoni09mmoqw}, in
contrast to this, energy-preserving csRK methods can be easily
constructed
\cite{brugnanoit10hbv,brugnanoit15aoh,hairer10epv,liw16ffe,quispelm08anc,
Tangs12ana,Tangs12tfe,Tangs14cor,miyatake14aee}. In addition, as
presented in \cite{Tangs12tfe,Tangsc15dgm}, some Galerkin
variational methods can be interpreted as continuous-stage (P)RK
methods, but they can not be completely understood in the classical
(P)RK framework. Therefore, continuous-stage methods have granted us
a new insight for numerical integration of differential equations
and some subjects in this new area need to be investigated.

Since symmetric integrators possess important theoretical and real
values in numerical ordinary differential equations
\cite{Cash05avs,hairerlw06gni,Mclachlanqt98nit,Stofferz95vsf}, we
are concerned with the development of new symmetric integrators for
solving second-order ordinary differential equations (ODEs). The
construction of such methods in this paper is on the basis of the
notion of csRKN methods and heavily relies on the Legendre
polynomial expansion technique. Furthermore, by using Gaussian and
Lobatto quadrature formulas we show that new families of symmetric
RKN-type schemes can be easily devised. Moreover, by Theorem
\ref{rever-preser}, these methods are also reversibility-preserving
and therefore very suitable for solving reversible systems.

This paper will be organized as follows. In Section 2, we introduce
the exact definition of csRKN methods for solving second-order ODEs
and the corresponding order theory previously developed in
\cite{Tangsz15hos} will be briefly revisited. In Section 3, by using
Legendre expansion technique, we present some useful results for
devising symmetric integrators which is then followed by giving some
illustrative examples for deriving new symmetric integrators in
Section 4. Some numerical experiments are reported in section 5. At
last, we give some concluding remarks in Section 6 to end this
paper.

\section{Continuous-stage RKN method and its order theory}

In this section, we will recall the notion of the so-called
continuous-stage Runge-Kutta-Nystr\"{o}m (csRKN) methods and review
some known results which are useful for constructing such methods of
arbitrarily high order. For more details, see
\cite{Tangz15spc,Tangsz15hos}.

\subsection{Continuous-stage RKN method}

Consider the following initial value problem governed by a
second-order system
\begin{align}\label{eq:second}
q''=f(t, q),\;q(t_0)=q_0,\,q'(t_0)=q'_0,
\end{align}
where $f: \mathbb{R}\times \mathbb{R}^d \rightarrow \mathbb{R}^d$ is
a smooth vector-valued function.

A well-known numerical method for solving \eqref{eq:second} is the
so-called RKN method with $s$ stages, which can be depicted as
\begin{subequations}
    \begin{alignat}{4}
    \label{eq:grkn1}
Q_i&=q_0 +hc_i q'_0 +h^2\sum\limits_{j=1}^s \bar{a}_{i j} f(t_0+c_j h, Q_j), \;i=1,\cdots, s, \\
    \label{eq:grkn2}
q_{1}&=q_0+ h q'_0+h^2\sum\limits_{i=1}^{s}\bar{b}_if(t_0+c_i h, Q_i),  \\
    \label{eq:grkn3}
q'_{1}& = q'_0 +h\sum\limits_{i=1}^{s} b_i f(t_0+c_i h, Q_i),
    \end{alignat}
\end{subequations}
and it can be characterized by the following Butcher tableau
\[\ba{c|c} c & \bar{A}\\[4pt]
\hline & \bar{b}  \\ \hline\\[-15pt] & b  \ea\]
where $\bar{A}=(\bar{a}_{i j})_{s\times
s},\;\bar{b}=(\bar{b}_1,\cdots,\bar{b}_s)^T,\;
b=(b_1,\cdots,b_s)^T,\;c=(c_1,\cdots,c_s)^T$. Compared with an
$s$-stage RK method applied to the corresponding first-order system
deduced from \eqref{eq:second}, the RKN method is preferable since
about half of the storage can be saved and the computational work
can be reduced a lot \cite{hairernw93sod}.

As a counterpart of the classical RKN method, the csRKN method can
be formally defined.

\begin{defn}\cite{Tangz15spc}\label{csRKN:def}
Let $\bar{A}_{\tau, \si}$ be a function of variables $\tau,
\si\in[0,1]$ and $\bar{B}_\tau,\;B_\tau,\;C_\tau$ be functions of
$\tau\in[0,1]$. For solving \eqref{eq:second}, the continuous-stage
Runge-Kutta-Nystr\"{o}m (csRKN) method as a one-step method mapping
$(q_0,q'_0)$ to $(q_1,q'_1)$ is given by
\begin{subequations}
    \begin{alignat}{2}
    \label{eq:csrkn1}
&Q_\tau=q_0 +hC_\tau q'_0 +h^2\int_{0}^{1} \bar{A}_{\tau, \si} f(t_0+C_\si h, Q_\si) \dif \si, \;\;\tau \in[0, 1], \\
    \label{eq:csrkn2}
&q_{1}=q_0+ h q'_0+h^2 \int_{0}^{1} \bar{B}_\tau  f(t_0+C_\tau h, Q_\tau) \dif \tau, \\
    \label{eq:csrkn3}
&q'_1 = q'_0 +h\int_{0}^{1} B_\tau f(t_0+C_\tau h, Q_\tau) \dif
\tau,
    \end{alignat}
\end{subequations}
which can be characterized by the following Butcher tableau
\[\ba{c|c} C_\tau & \bar{A}_{\tau,\sigma} \\[4pt]
\hline & \bar{B}_\tau  \\ \hline\\[-15pt] & B_\tau  \ea\]
\end{defn}

\subsection{Order theory for RKN-type method}

\begin{defn}\cite{hairernw93sod}
A RKN-type method is of order $p$, if for all regular problem
\eqref{eq:second}, the following two formulas hold, as
$h\rightarrow0$,
\begin{equation*}
q(t_0+h)-q_1=\mathcal{O}(h^{p+1}),\quad
q'(t_0+h)-q'_1=\mathcal{O}(h^{p+1}).
\end{equation*}
\end{defn}

We introduce the following classical simplifying assumptions for RKN
methods \cite{hairernw93sod,hairerlw06gni}
\begin{equation}\label{RKN-simpl-assump}
\begin{split}
&B(\xi):\; \sum_{i=1}^sb_ic_i^{\kappa-1}=\frac{1}{\kappa},\;1\leq\kappa\leq\xi,\\
&CN(\eta):\;\sum_{j=1}^s\bar{a}_{ij}c_j^{\kappa-1}=\frac{c_i^{\kappa+1}}{\kappa(\kappa+1)},\;1\leq i\leq s,\,1\leq\kappa\leq\eta-1,\\
&DN(\zeta):\;\sum_{i=1}^sb_ic_i^{\kappa-1}\bar{a}_{ij}=
\frac{b_jc_j^{\kappa+1}}{\kappa(\kappa+1)}-\frac{b_jc_j}{\kappa}+\frac{b_j}{\kappa+1},\;1\leq
j\leq s,\,1\leq\kappa\leq\zeta-1.
\end{split}
\end{equation}
\begin{thm}\cite{hairernw93sod}\label{ord_RKN}
If the coefficients of the RKN method
\eqref{eq:grkn1}-\eqref{eq:grkn3} satisfy the simplifying
assumptions $B(p),\,CN(\eta),\,DN(\zeta)$, and if
$\bar{b}_i=b_i(1-c_i)$ holds for all $i=1,\ldots, s$, then the
method is of order at least $\min\{p,\,2\eta+2,\eta+\zeta\}$.
\end{thm}

Analogously to the classical case, we have the following simplifying
assumptions for csRKN methods \cite{Tangsz15hos}
\begin{equation*}\label{csRKN-simpl-assump}
\begin{split}
&\mathcal{B}(\xi):\quad \int_0^1B_\tau C_\tau^{\kappa-1}\,\dif
\tau=\frac{1}{\kappa},\quad 1\leq\kappa\leq\xi,\\
&\mathcal{CN}(\eta):\quad
\int_0^1\bar{A}_{\tau,\,\sigma}C_\sigma^{\kappa-1}\,\dif
\sigma=\frac{C_\tau^{\kappa+1}}{\kappa(\kappa+1)},\quad \forall\;\tau\in[0,1],\;1\leq\kappa\leq\eta-1,\\
&\mathcal{DN}(\zeta):\quad \int_0^1B_\tau C_\tau^{\kappa-1}
\bar{A}_{\tau,\,\sigma}\,\dif \tau=\frac{B_\sigma
C_\sigma^{\kappa+1}}{\kappa(\kappa+1)}-\frac{B_\sigma
C_\sigma}{\kappa} +\frac{B_\sigma}{\kappa+1},\quad
\forall\;\si\in[0,1],\; 1\leq\kappa\leq\zeta-1.
\end{split}
\end{equation*}
\begin{thm}\cite{Tangsz15hos}\label{ord_csRKN}
If the coefficients of the csRKN method
\eqref{eq:csrkn1}-\eqref{eq:csrkn3} satisfy the simplifying
assumptions
$\mathcal{B}(p),\,\mathcal{CN}(\eta),\,\mathcal{DN}(\zeta)$, and if
$\bar{B}_\tau=B_\tau(1-C_\tau)$ holds for $\tau\in[0,1]$, then the
method is of order at least $\min\{p,\,2\eta+2,\eta+\zeta\}$.
\end{thm}

Let us introduce the normalized shifted Legendre polynomial $P_k(x)$
of degree $k$ by the following Rodrigues' formula
\begin{equation*}
P_0(x)=1,\;P_k(x)=\frac{\sqrt{2k+1}}{k!}\frac{{\dif}^k}{\dif x^k}
[(x^2-x)^k],\; \;k=1,2,3,\cdots.
\end{equation*}
A well-known property of Legendre polynomials is that they are
orthogonal to each other with respect to the $L^2$ inner product in
$[0,\,1]$
\begin{equation*}
\int_0^1 P_j(x) P_k(x)\,\dif x= \delta_{jk},\quad j,\,
k=0,1,2,\cdots,
\end{equation*}
where $\delta_{jk}$ is the Kronecker delta. For convenience, we list
some of them as follows
\begin{equation*}
P_0(x)=1,\;P_1(x)=\sqrt{3}(2x-1),\;P_2(x)=\sqrt{5}(6x^2-6x+1),\,\cdots.
\end{equation*}
\begin{thm}\cite{Tangsz15hos}\label{thm:csRKN}
For the csRKN method \eqref{eq:csrkn1}-\eqref{eq:csrkn3} denoted by
$(\bar{A}_{\tau,\si},\bar{B}_\tau,B_\tau,C_\tau)$ with the
assumption $B_\tau=1, C_\tau=\tau$,  the following two statements
are equivalent to each other:
\begin{itemize}
\item[(I)] both $\mathcal{CN}(\eta)$ and $\mathcal{DN}(\zeta)$ hold true;
\item[(II)] $\bar{A}_{\tau,\,\sigma}$
possesses the following form in terms of Legendre polynomials
\begin{equation}\label{expan6}
\begin{split}
\bar{A}_{\tau,\,\si}&=\frac{1}{6}-\frac{1}{2}\xi_1P_1(\si)+\frac{1}{2}\xi_1P_1(\tau)
+\sum_{\iota=1}^{N_1}\xi_{\iota}\xi_{\iota+1}
P_{\iota-1}(\tau)P_{\iota+1}(\si)\\
&-\sum_{\iota=1}^{N_2}\big(\xi_{\iota}^2+\xi_{\iota+1}^2\big)
P_{\iota}(\tau)P_{\iota}(\sigma)+\sum_{\iota=1}^{N_3}\xi_{\iota}\xi_{\iota+1}
P_{\iota+1}(\tau)P_{\iota-1}(\sigma)\\
&+\sum_{i\geq\zeta-1 \atop
j\geq\eta-1}\omega_{(i,\,j)}P_i(\tau)P_j(\si).
\end{split}
\end{equation}
where
$\xi_\iota=\frac{1}{2\sqrt{4\iota^2-1}},\,N_1=\max\{\eta-3,\,\zeta-1\},\,N_2=\max\{\eta-2,\,\zeta-2\},
\,N_3=\max\{\eta-1,\,\zeta-3\}$ and $\omega_{(i,\,j)}$ are arbitrary
real numbers.
\end{itemize}
\end{thm}

Recall that we have $\mathcal{B}(\infty)$ by using $B_\tau=1,
C_\tau=\tau$, thus Theorem \ref{thm:csRKN} implies that we can
easily construct a csRKN method with order
$\min\{\infty,\,2\eta+2,\eta+\zeta\}=\min\{2\eta+2,\eta+\zeta\}$ (by
Theorem \ref{ord_csRKN}). However, for the sake of deriving a
practical csRKN method, we need to define a finite form for the
coefficient $\bar{A}_{\tau,\,\si}$, which can be easily realized by
truncating the series \eqref{expan6}. In such a case, we get
$\bar{A}_{\tau,\,\si}$ which is a bivariate polynomial.
Consequently, by applying a quadrature formula denoted by $(b_i,
c_i)_{i=1}^s$ to \eqref{eq:csrkn1}-\eqref{eq:csrkn3}, it leads to an
$s$-stage RKN method
\begin{subequations}
    \begin{alignat}{2}
    \label{eq:rkn1}
&Q_i=q_0 +hC_{c_i} q'_0 +h^2\sum\limits_{j=1}^{s} b_j \bar{A}_{c_i, c_j} f(t_0+C_{c_j}h,\,Q_j), \quad i=1,\cdots, s, \\
    \label{eq:rkn2}
&q_{1}=q_0+ h q'_0+h^2\sum\limits_{i=1}^{s} b_i\bar{B}_{c_i}  f(t_0+C_{c_i}h,\,Q_i),  \\
    \label{eq:rkn3}
&q'_{1}=q'_0+h\sum\limits_{i=1}^{s}b_iB_{c_i}f(t_0+C_{c_i}h,\,Q_i),
    \end{alignat}
\end{subequations}
whose Butcher tableau is
\begin{equation}\label{RKN:qua_orig}
\ba{c|ccc} C_{c_1} & b_1\bar{A}_{c_1, c_1}
& \cdots & b_s\bar{A}_{c_1, c_s}\\[2pt]
\vdots &\vdots &\vdots\\[2pt]
C_{c_s} & b_1\bar{A}_{c_s, c_1} &
\cdots & b_s\bar{A}_{c_s, c_s}\\[2pt]
\hline & b_1\bar{B}_{c_1}  & \cdots & b_s\bar{B}_{c_s}\\[2pt]
\hline & b_1B_{c_1}  & \cdots & b_sB_{c_s}\ea
\end{equation}

If we additionally assume
$\bar{B}_\tau=B_\tau(1-C_\tau),\,B_\tau=1,\, C_\tau=\tau$, then it
gives an $s$-stage RKN method with tableau
\begin{equation}\label{RKN:qua}
\ba{c|ccc} c_1 & b_1\bar{A}_{c_1, c_1}
& \cdots & b_s\bar{A}_{c_1, c_s}\\[2pt]
\vdots &\vdots &\vdots\\[2pt]
c_s & b_1\bar{A}_{c_s, c_1} &
\cdots & b_s\bar{A}_{c_s, c_s}\\[2pt]
\hline & \bar{b}_1  & \cdots & \bar{b}_s\\[2pt]
\hline & b_1  & \cdots & b_s\ea
\end{equation}
where $\bar{b}_i=b_i(1-c_i),\; i=1,\cdots,s$.

In view of Theorem \ref{ord_RKN}, we have the following result for
analyzing the order of the RKN method with tableau \eqref{RKN:qua}.

\begin{thm}\cite{Tangsz15hos}\label{qua:csRKN}
Assume $\bar{A}_{\tau,\,\sigma}$ is a bivariate polynomial of degree
$\pi_A^{\tau}$ in $\tau$ and degree $\pi_A^{\sigma}$ in $\sigma$,
and the quadrature formula $(b_i,c_i)_{i=1}^s$ is of order $p$. If
the coefficients of the underlying csRKN method
\eqref{eq:csrkn1}-\eqref{eq:csrkn3} satisfy
$\bar{B}_\tau=B_\tau(1-C_\tau),\,B_\tau=1,\, C_\tau=\tau$, and both
$\mathcal{CN}(\eta)$, $\mathcal{DN}(\zeta)$ hold true, then the RKN
method with tableau \eqref{RKN:qua} is of order at least
$$\min(p, \,2\alpha+2, \,\alpha+\beta),$$
where $\alpha=\min(\eta,\,p-\pi_A^{\sigma}+1)$ and
$\beta=\min(\zeta,\, p-\pi_A^{\tau}+1)$.
\end{thm}

\section{Conditions for the symmetry of csRKN methods}

Now let us introduce the definition of symmetric methods and then
show the conditions for a csRKN method to be symmetric.

\begin{defn}\cite{hairerlw06gni}
A numerical one-step method $\Phi_h$ is called symmetric if it
satisfies
$$\Phi^*_h=\Phi_h,$$
where $\Phi^*_h=\Phi^{-1}_{-h}$ is referred to as the adjoint method
of $\Phi_h$.
\end{defn}
Symmetry implies that the original method and the adjoint method
give identical numerical results. An attractive property of
symmetric integrators is that they possess an \emph{even order}
\cite{hairerlw06gni}. By definition, a one-step method
$z_1=\Phi_h(z_0; t_0,t_1)$ is symmetric if exchanging
$h\leftrightarrow -h$, $z_0\leftrightarrow z_1$ and
$t_0\leftrightarrow t_1$ leaves the original method unaltered.

\begin{thm}\label{symm_cond_ori}
If the coefficients of the csRKN method
\eqref{eq:csrkn1}-\eqref{eq:csrkn3} satisfy
\begin{equation}\label{sym_conds01}
\begin{split}
C_\tau&=1-C_{1-\tau},\\
\bar{A}_{\tau,\si}&=B_{1-\si}(1-C_{1-\tau})-\bar{B}_{1-\si}+\bar{A}_{1-\tau,1-\si},\\
\bar{B}_\tau&=B_{1-\tau}-\bar{B}_{1-\tau},\\
B_\tau&=B_{1-\tau},
\end{split}
\end{equation}
for $\forall\,\tau,\,\si\in[0,1]$, then the method is symmetric.
\end{thm}
\begin{proof}
Firstly, let us establish the adjoint method. From
\eqref{eq:csrkn1}-\eqref{eq:csrkn3}, by interchanging $t_0, q_0,
q'_0, h$ with $t_1, q_1, q'_1, -h$ respectively, we have
\begin{subequations}
    \begin{alignat}{3}
     \label{eq:rever1}
&Q_\tau=q_1 -hC_\tau q'_1 +h^2\int_{0}^{1} \bar{A}_{\tau, \si} f(t_1-C_\si h, Q_\si) \dif \si, \;\;\tau \in[0, 1], \\
     \label{eq:rever2}
&q_{0}=q_1- h q'_1+h^2 \int_{0}^{1} \bar{B}_\tau  f(t_1-C_\tau h, Q_\tau) \dif\tau, \\
     \label{eq:rever3}
&q'_0 = q'_1-h\int_{0}^{1} B_\tau f(t_1-C_\tau h, Q_\tau) \dif\tau.
   \end{alignat}
\end{subequations}
Notice that $t_1-C_\tau h=t_0+(1-C_\tau)h$, thus \eqref{eq:rever3}
becomes
\begin{equation}\label{recast01}
q'_1=q'_0+h\int_{0}^{1} B_\tau f(t_0+(1-C_\tau)h, Q_\tau) d\tau.
\end{equation}
Substituting it into \eqref{eq:rever2} yields
\begin{equation}\label{recast02}
q_1=q_{0}+hq'_0+h^2 \int_{0}^{1} (B_\tau-\bar{B}_\tau)
f(t_0+(1-C_\tau)h, Q_\tau)\dif\tau.
\end{equation}
Next, by inserting \eqref{recast01} and \eqref{recast02} into
\eqref{eq:rever1}, it follows that
\begin{equation}\label{recast03}
Q_\tau=q_0+h(1-C_\tau)q'_0+h^2\int_{0}^{1}(B_\si(1-C_\tau)-\bar{B}_\si+\bar{A}_{\tau,
\si}) f(t_0+(1-C_\si) h, Q_\si) \dif \si.
\end{equation}
By replacing $\tau$ and $\si$ with $1-\tau$ and $1-\si$
respectively, we can recast \eqref{recast03}, \eqref{recast02} and
\eqref{recast01} as
\begin{equation}\label{adjo}
\begin{split}
&Q^*_\tau=q_0+hC^*_\tau q'_0 +h^2\int_{0}^{1}\bar{A}^*_{\tau, \si} f(t_0+C^*_\si h, Q^*_\si) \dif\si, \;\;\tau \in[0, 1], \\
&q_{1}=q_0+ h q'_0+h^2 \int_{0}^{1} \bar{B}^*_\tau  f(t_0+C^*_\tau h, Q^*_\tau) \dif\tau, \\
&q'_1 = q'_0 +h\int_{0}^{1} B^*_\tau f(t_0+C^*_\tau h, Q^*_\tau)
\dif\tau,
\end{split}
\end{equation}
where $Q^*_\tau=Q_{1-\tau},\,\tau\in[0,1]$ and
\begin{equation}\label{adjo:coe}
\begin{split}
C^*_\tau&=1-C_{1-\tau},\\
\bar{A}^*_{\tau,\si}&=B_{1-\si}(1-C_{1-\tau})-\bar{B}_{1-\si}+\bar{A}_{1-\tau,
1-\si},\\
\bar{B}^*_\tau&=B_{1-\tau}-\bar{B}_{1-\tau},\\
B^*_\tau&=B_{1-\tau},
\end{split}
\end{equation}
for $\forall\,\tau,\,\si\in[0,1]$. Therefore, we have get the
adjoint method defined by \eqref{adjo} and \eqref{adjo:coe}. Given
that a csRKN method can be uniquely determined by its coefficients,
hence if we require the following condition
$$C_\tau=C^*_\tau,\,\bar{A}_{\tau,\si}=\bar{A}^*_{\tau,\si},\,\bar{B}_\tau=\bar{B}^*_\tau,\,B_\tau=B^*_\tau,$$
namely the condition \eqref{sym_conds01}, then the original method
is symmetric.
\end{proof}
In the following we present a preferable result for ease of devising
symmetric csRKN methods.
\begin{thm}\label{constr_symmcsRKN}
Suppose that $\bar{B}_\tau=B_\tau(1-C_\tau),\,B_\tau=1,\,
C_\tau=\tau$, then the csRKN method denoted by
$(\bar{A}_{\tau,\si},\bar{B}_\tau,B_\tau,C_\tau)$ is symmetric, if
$\bar{A}_{\tau,\si}$ possesses the following form in terms of
Legendre polynomials
\begin{equation}\label{sym_conds02}
\begin{split}
&\bar{A}_{\tau,\si}=\alpha_{(0,0)}-\frac{1}{2}\xi_1P_1(\si)
+\frac{1}{2}\xi_1P_1(\tau)+\sum\limits_{i+j\,\text{is}\,\text{even}
\atop i+j>1}\alpha_{(i,j)}
P_i(\tau)P_j(\sigma),\quad\tau,\si\in[0,1],
\end{split}
\end{equation}
where $\xi_\iota=\frac{1}{2\sqrt{4\iota^2-1}}$ and $\alpha_{(i,j)}$
are arbitrary real numbers.
\end{thm}
\begin{proof}
By noticing $\bar{B}_\tau=B_\tau(1-C_\tau),\,B_\tau=1,\,
C_\tau=\tau$, it suffices for us to consider the second condition
given in \eqref{sym_conds01}. By using a simple identity
$\tau=\frac{1}{2}P_0(\tau)+\xi_1P_1(\tau)$, it implies
\begin{equation}\label{eq:A}
\bar{A}_{\tau,\,\sigma}-\bar{A}_{1-\tau,\,1-\sigma}=\tau-\sigma=
\xi_1(P_1(\tau)-P_1(\sigma)).
\end{equation}

Next, let us consider the following expansion of
$\bar{A}_{\tau,\,\sigma}$ in terms of the Legendre orthogonal basis
$\{P_i(\tau)P_j(\sigma):\,i,\,j\geq0\}$,
\begin{equation*}
\bar{A}_{\tau,\,\sigma}=\sum\limits_{i,j\geq0}\alpha_{(i,j)}
P_i(\tau)P_j(\sigma),\quad\alpha_{(i,j)}\in \mathbb{R},
\end{equation*}
and then by replacing $\tau$ and $\si$ with $1-\tau$ and $1-\si$
respectively, with the help of $P_\iota(1-t)=(-1)^\iota
P_\iota(t)\,(\iota\geq0)$, we have
\begin{equation*}
\bar{A}_{1-\tau,\,1-\sigma}=\sum\limits_{i,j\geq0}(-1)^{i+j}\alpha_{(i,j)}P_i(\tau)P_j(\sigma).
\end{equation*}
Substituting the above two expressions into \eqref{eq:A} and
collecting the like basis, follows
$$\alpha_{(0,1)}=-\frac{1}{2}\xi_1,\,\alpha_{(1,0)}=\frac{1}{2}\xi_1,
\;\,\alpha_{(i,j)}=0,\;\text{when}\;i+j\;\text{is}\;\text{odd}\;\text{and}\;
i+j>1,$$ which completes the proof.
\end{proof}

By putting Theorem \ref{thm:csRKN} and Theorem
\ref{constr_symmcsRKN} together, we can devise symmetric integrators
of arbitrarily high order. Besides, as an alternative way, we can
use the same technique as presented in \cite{Tangz15spc} to
construct symmetric integrators for arbitrary order, that is,
substituting \eqref{sym_conds02} into the order conditions (see
\cite{Tangz15spc}, Page 12) one by one and determining the
corresponding parameters $\alpha_{(i,j)}$. As symmetric methods
possess an even order, it is sufficient to consider those order
conditions for odd orders, so we can increase two orders per step.
We present the the following result without a proof (please see
\cite{Tangz15spc} for a similar proof).
\begin{thm}
Suppose that $\bar{A}_{\tau,\si}$ is in the form \eqref{sym_conds02}
and $\bar{B}_\tau=B_\tau(1-C_\tau),\,B_\tau=1,\, C_\tau=\tau$. Then
the corresponding csRKN method is symmetric and of order $2$ at
least. If we additionally require $\alpha_{(0,0)}=\frac{1}{6}$, then
the method is of order $4$ at least. Moreover, if we further require
that
\begin{equation}\label{coef:sixth}
\begin{split}
&\alpha_{(0,0)}=\frac{1}{6},\,\alpha_{(1,1)}=-\frac{1}{10},
\,\alpha_{(2,0)}=\alpha_{(0,2)}=\frac{\sqrt{5}}{60},\\
&\alpha_{(i,0)}=0,\quad \text{for}\;\text{even}\;i>2,
\end{split}
\end{equation}
then the method is of order $6$ at least.
\end{thm}

\section{Symmetric RKN method}

In this section, we show that symmetric RKN methods can be easily
derived from symmetric csRKN methods by using quadrature formulas.

\begin{thm}\label{symm_quad}
If the coefficients of the underlying symmetric csRKN method satisfy
\eqref{sym_conds01}, then the associated RKN method
\eqref{RKN:qua_orig} is symmetric, provided that the weights and
abscissae of the quadrature formula satisfy $b_{s+1-i}=b_i$ and
$c_{s+1-i}=1-c_i$ for all $i$.
\end{thm}
\begin{proof}
The symmetric condition for an $s$-stage classical RKN method
denoted by $(\bar{a}_{ij},\,\bar{b}_i,\,b_i,\,c_i)$ is known as
(see, e.g., \cite{okunbors92ecm})
\begin{equation*}
\begin{split}
c_i&=1-c_{s+1-i},\\
\bar{a}_{ij}&=b_{s+1-j}(1-c_{s+1-i})-\bar{b}_{s+1-j}+\bar{a}_{s+1-i,s+1-j},\\
\bar{b}_i&=b_{s+1-i}-\bar{b}_{s+1-i},\\
b_i&=b_{s+1-i},
\end{split}
\end{equation*}
for all $i, j=1,\cdots,s$. By using \eqref{sym_conds01}, we have
\begin{equation*}
\begin{split}
C_{c_i}&=1-C_{1-c_i},\\
\bar{A}_{c_i,c_j}&=B_{1-c_j}(1-C_{1-c_i})-\bar{B}_{1-c_j}+\bar{A}_{1-c_i,1-c_j},\\
\bar{B}_{c_i}&=B_{1-c_i}-\bar{B}_{1-c_i},\\
B_{c_i}&=B_{1-c_i},
\end{split}
\end{equation*}
for all $i, j=1,\cdots,s$. In view of $b_{s+1-i}=b_i$ and
$c_{s+1-i}=1-c_i$ for all $i$,  the coefficients
$(b_j\bar{A}_{c_i,c_j},\,b_i\bar{B}_{c_i},\,b_iB_{c_i},\,C_i)$ of
the associated RKN method satisfy
\begin{equation*}
\begin{split}
C_{c_i}&=1-C_{c_{s+1-i}},\\
b_j\bar{A}_{c_i,c_j}&=b_{s+1-j}B_{c_{s+1-j}}(1-C_{c_{s+1-i}})\\
&\quad-b_{s+1-j}\bar{B}_{c_{s+1-j}}+b_{s+1-j}\bar{A}_{c_{s+1-i},c_{s+1-j}},\\
b_i\bar{B}_{c_i}&=b_{s+1-i}B_{c_{s+1-i}}-b_{s+1-i}\bar{B}_{c_{s+1-i}},\\
b_iB_{c_i}&=b_{s+1-i}B_{c_{s+1-i}},
\end{split}
\end{equation*}
for all $i, j=1,\cdots,s$, which completes the proof by the
classical result.
\end{proof}
\begin{cor}
If $\bar{A}_{\tau,\si}$ takes the form \eqref{sym_conds02} and
$\bar{B}_\tau=B_\tau(1-C_\tau),\,B_\tau=1,\, C_\tau=\tau$, then by
using a quadrature formula $(b_i, c_i)_{i=1}^s$ with $b_{s+1-i}=b_i$
and $c_{s+1-i}=1-c_i$ for all $i$, the resulting RKN method
\eqref{RKN:qua} is symmetric.
\end{cor}

Since the weights and abscissae of Gaussian-type and Lobatto-type
quadrature formulas satisfy $b_{s+1-i}=b_i$ and $c_{s+1-i}=1-c_i$
for all $i$, they can be used for devising symmetric RKN methods.

\begin{table}
\[\ba{c|c} \frac{1}{2} & \alpha\\[2pt]
\hline & \frac{1}{2}\\[2pt]
\hline & 1\ea\qquad \ba{c|cc} 0 & \frac{1}{2}\alpha&
\frac{1}{2}\alpha-\frac{1}{4}\\[2pt]
1 & \frac{1}{2}\alpha+\frac{1}{4}&
\frac{1}{2}\alpha\\[2pt]
\hline & \frac{1}{2}& 0\\[2pt]
\hline & \frac{1}{2}& \frac{1}{2} \ea\] \caption{Two families of
symmetric and symplectic RKN methods of order 2, by using Gaussian
(on the left) and Lobatto (on the right) quadrature formulas
respectively.}\label{exa:SmmRKN2}
\end{table}
\begin{exa}
If we take the coefficients
$(\bar{A}_{\tau,\si},\,\bar{B}_\tau,\,B_\tau,\,C_\tau)$ as
\begin{equation} \label{eq:2coeff}
\begin{split}
&\bar{A}_{\tau,\si}=\alpha-\frac{\sqrt{3}}{12}P_1(\si)+\frac{\sqrt{3}}{12}P_1(\tau),\\
&\bar{B}_\tau=1-\tau,\;B_\tau=1,\;C_\tau=\tau,
\end{split}
\end{equation}
with one parameter $\alpha$ being introduced, then we get a family
of symmetric csRKN methods with order $2$.  By Theorem $3.3$
presented in \cite{Tangz15spc} (see also Theorem $4.4$ in
\cite{Tangsz15hos}), such methods are also symplectic and thus
suitable for solving general second-order Hamiltonian systems.

By using suitable quadrature formulas with order $p\geq2$ we can get
symmetric RKN methods of order\footnote{This can be easily checked
by the classical order conditions that listed in
\cite{hairernw93sod} (see also \cite{Tangz15spc}).} $2$. The
resulting symmetric RKN methods are shown in Table
\ref{exa:SmmRKN2}.
\end{exa}

\begin{table}
    \[\ba{c|cc} \frac{3-\sqrt{3}}{6}&
    \frac{1+6\alpha}{12}&
    \frac{1-\sqrt{3}-6\alpha}{12}\\[2pt]
    \frac{3+\sqrt{3}}{6} & \frac{1+\sqrt{3}-6\alpha}{12}&
    \frac{1+6\alpha}{12}\\[2pt]
    \hline & \frac{3+\sqrt{3}}{12}& \frac{3-\sqrt{3}}{12}\\[2pt]
    \hline & \frac{1}{2}& \frac{1}{2} \ea \qquad \ba{c|ccc} 0 &
    \frac{1+18\alpha+6\sqrt{5}(\beta+\gamma)}{36} &
    \frac{-1-6\sqrt{5}(\beta-2\gamma)}{18} & \frac{-1-9\alpha+3\sqrt{5}(\beta+\gamma)}{18} \\[2pt]
    \frac{1}{2} & \frac{5+6\sqrt{5}(2\beta-\gamma)}{72}  &
    \frac{1-3\sqrt{5}(\beta+\gamma)}{9}&\frac{-1+6\sqrt{5}(2\beta-\gamma)}{72}\\[2pt]
    1 & \frac{2-9\alpha+3\sqrt{5}(\beta+\gamma)}{18}& \frac{5-6\sqrt{5}(\beta-2\gamma)}{18}
    &\frac{1+18\alpha+6\sqrt{5}(\beta+\gamma)}{36}\\[2pt]
    \hline & \frac{1}{6}& \frac{1}{3}& 0\\[2pt]
    \hline & \frac{1}{6}& \frac{2}{3}& \frac{1}{6} \ea \]
    \caption{Two families of symmetric RKN methods of order 4, by using Gaussian (2 nodes) and Lobatto (3
    nodes) quadrature formulae.}\label{exa:SmmRKN4}
\end{table}

\begin{exa}
If we take the coefficients
$(\bar{A}_{\tau,\si},\,\bar{B}_\tau,\,B_\tau,\,C_\tau)$ as
\begin{equation}\label{eq:4coeff}
\begin{split}
&\bar{A}_{\tau,\si}=\frac{1}{6}-\frac{\sqrt{3}}{12}P_1(\si)+\frac{\sqrt{3}}{12}P_1(\tau)+\alpha
P_1(\tau)P_1(\si)+\beta P_0(\tau)P_2(\si)+\gamma  P_0(\si)P_2(\tau),\\
&\bar{B}_\tau=1-\tau,\;B_\tau=1,\;C_\tau=\tau,
\end{split}
\end{equation}
then we get a family of symmetric csRKN methods with order $4$. By
using suitable quadrature formulas with order $p\geq4$ we get
symmetric RKN methods of order $4$, which are shown in Table
\ref{exa:SmmRKN4}.
\end{exa}

\begin{rem} We point out that:
    \begin{itemize}
        \item[(1)] The left family of RKN methods in Table
\ref{exa:SmmRKN4} are always symmetric
        and symplectic, while the right family of RKN methods of Table
\ref{exa:SmmRKN4} are symmetric and symplectic when $\beta=\gamma$.
        \item[(2)] The classical 3-stage Lobatto IIIA method \cite{hairerlw06gni} induces
        the following RKN method,
        \begin{equation}\label{LobattoIIIA}
        \ba{c|ccc} 0 &
        0 &
        0 &0 \\[2pt]
        \frac{1}{2} & \frac{1}{16}  &
        \frac{1}{12}&-\frac{1}{48}\\[2pt]
        1 & \frac{1}{6}& \frac{1}{3}
        &0\\[2pt]
        \hline & \frac{1}{6}& \frac{1}{3}& 0\\[2pt]
        \hline & \frac{1}{6}& \frac{2}{3}& \frac{1}{6} \ea
        \end{equation}
        which can be retrieved  by taking
        $\alpha=-\frac{1}{12},\beta=0,\gamma=\frac{\sqrt{5}}{60}$ in
        Table \ref{exa:SmmRKN4}.
        \item[(3)] The classical 3-stage Lobatto IIIB method  \cite{hairerlw06gni} induces
        the following RKN method,
        \begin{equation}\label{LobattoIIIB}
        \ba{c|ccc} 0 &
        0 &
        -\frac{1}{12} &0 \\[2pt]
        \frac{1}{2} & \frac{1}{12}  &
        \frac{1}{12}&0\\[2pt]
        1 & \frac{1}{6}& \frac{1}{4}
        &0\\[2pt]
        \hline & \frac{1}{6}& \frac{1}{3}& 0\\[2pt]
        \hline & \frac{1}{6}& \frac{2}{3}& \frac{1}{6} \ea
        \end{equation}
        which can be retrieved  by taking $\alpha=-\frac{1}{12},\beta=\frac{\sqrt{5}}{60},\gamma=0$ in
        Table \ref{exa:SmmRKN4}.
        \end{itemize}
\end{rem}

\begin{table}
\[\ba{c|ccc} \frac{5-\sqrt{15}}{10} & \frac{2+60\alpha}{135} &
\frac{19-6\sqrt{15}-120\alpha}{270}&
\frac{62-15\sqrt{15}+120\alpha}{540}\\[2pt]
\frac{1}{2} & \frac{19+6\sqrt{15}-120\alpha}{432}&
\frac{1+15\alpha}{27}&
\frac{19-6\sqrt{15}-120\alpha}{432}\\[2pt]
\frac{5+\sqrt{15}}{10}& \frac{62+15\sqrt{15}+120\alpha}{540}
&\frac{19+6\sqrt{15}-120\alpha}{270}&\frac{2+60\alpha}{135}\\[2pt]
\hline & \frac{5+\sqrt{15}}{36} & \frac{2}{9} & \frac{5-\sqrt{15}}{36}\\[2pt]
\hline & \frac{5}{18} & \frac{4}{9} & \frac{5}{18}
\ea\]\\
\[\ba{c|cccc}0 &
\frac{1+150\alpha}{360}&\frac{-5-3\sqrt{5}-300\alpha}{720}
& \frac{-5+3\sqrt{5}-300\alpha}{720}& \frac{2+75\alpha}{180}\\[2pt]
\frac{5-\sqrt{5}}{10}&\frac{29}{720}-\frac{11\sqrt{5}+100\alpha}{1200}&\frac{11+30\alpha}{360}
&\frac{29-15\sqrt{5}+30\alpha}{360}& -\frac{1}{720}+\frac{\sqrt{5}-100\alpha}{1200}\\[2pt]
\frac{5+\sqrt{5}}{10}&\frac{29}{720}+\frac{11\sqrt{5}-100\alpha}{1200}&\frac{29+15\sqrt{5}+30\alpha}{360}
&\frac{11+30\alpha}{360}&-\frac{1}{720}-\frac{\sqrt{5}+100\alpha}{1200}\\[2pt]
1&\frac{17+75\alpha}{180}&\frac{145+33\sqrt{5}-300\alpha}{720}&\frac{145-33\sqrt{5}-300\alpha}{720}
& \frac{1+150\alpha}{360}\\[2pt]
\hline & \frac{1}{12}& \frac{5+\sqrt{5}}{24}&\frac{5-\sqrt{5}}{24}& 0\\[2pt]
\hline & \frac{1}{12}& \frac{5}{12}& \frac{5}{12}&\frac{1}{12}\ea\]
\caption{Two families of symmetric and symplectic RKN methods of
order 6, by using Gaussian (on the top) and Lobatto (on the bottom)
quadrature formulas respectively.}\label{exa:SmmRKN6}
\end{table}

\begin{exa}
If we take the coefficients
$(\bar{A}_{\tau,\si},\,\bar{B}_\tau,\,B_\tau,\,C_\tau)$ as
\begin{equation}\label{eq:5coeff}
\begin{split}
&\bar{A}_{\tau,\si}=\sum\limits_{i+j\leq2}\alpha_{(i,j)}P_i(\tau)P_j(\sigma)
+\alpha P_2(\tau)P_2(\sigma),\\
&\bar{B}_\tau=1-\tau,\;B_\tau=1,\;C_\tau=\tau,
\end{split}
\end{equation}
where
$\alpha_{(0,1)}=-\frac{\sqrt{3}}{12},\,\alpha_{(1,0)}=\frac{\sqrt{3}}{12},$
and the remaining $\alpha_{(i,j)}$ satisfy \eqref{coef:sixth}, then
we get a family of $6$-order symmetric and symplectic csRKN methods.
By using suitable quadrature formulas with order $p\geq6$ we get
symmetric and symplectic RKN methods of order $6$, which are shown
in Table \ref{exa:SmmRKN6}.
\end{exa}

\section{Numerical experiments}

In this section, we perform some numerical results for comparing the
numerical behaviors of the presented methods. For this aim, we
consider the $4$-order method \eqref{LobattoIIIB} and the following
three $4$-order methods:
\begin{itemize}
  \item By taking $\alpha=0,\beta=\gamma=\frac{\sqrt{5}}{30}$ in Table
\ref{exa:SmmRKN4} it leads to a diagonally implicit symplectic and
symmetric RKN method
        \begin{equation}\label{diagsym}
        \ba{c|ccc} 0 &
        \frac{1}{12} &
        0 &0 \\[2pt]
        \frac{1}{2} & \frac{1}{12}  &
        0&0\\[2pt]
        1 & \frac{1}{6}& \frac{1}{3}
        &\frac{1}{12}\\[2pt]
        \hline & \frac{1}{6}& \frac{1}{3}& 0\\[2pt]
        \hline & \frac{1}{6}& \frac{2}{3}& \frac{1}{6} \ea
        \end{equation}
  \item By taking
$\alpha=-\frac{1}{10},\beta=\frac{\sqrt{5}}{150},\gamma=\frac{\sqrt{5}}{60}$
in Table \ref{exa:SmmRKN4} it gives the following symmetric RKN
method
        \begin{equation}\label{rkna}
        \ba{c|ccc} 0 &
        -\frac{1}{360} &
        -\frac{1}{90} & \frac{1}{72} \\[2pt]
        \frac{1}{2} & \frac{49}{720}  &
        \frac{13}{180}& -\frac{11}{720}\\[2pt]
        1 & \frac{13}{72}
        &\frac{29}{90}&-\frac{1}{360}\\[2pt]
        \hline & \frac{1}{6}& \frac{1}{3}& 0\\[2pt]
        \hline & \frac{1}{6}& \frac{2}{3}& \frac{1}{6} \ea
        \end{equation}
  \item By taking
$\alpha=-\frac{1}{10},\beta=\frac{\sqrt{5}}{60},\gamma=\frac{\sqrt{5}}{150}$
in Table \ref{exa:SmmRKN4} it gives the following symmetric RKN
method
        \begin{equation}\label{rknb}
        \ba{c|ccc} 0 &
        -\frac{1}{360} &
        -\frac{11}{180} & \frac{1}{72} \\[2pt]
        \frac{1}{2} & \frac{29}{360}  &
        \frac{13}{180}& -\frac{1}{360}\\[2pt]
        1 & \frac{13}{72}
        &\frac{49}{180}&-\frac{1}{360}\\[2pt]
        \hline & \frac{1}{6}& \frac{1}{3}& 0\\[2pt]
        \hline & \frac{1}{6}& \frac{2}{3}& \frac{1}{6} \ea
        \end{equation}
\end{itemize}

For convenience, we denote four symmetric RKN methods
\eqref{LobattoIIIB}, \eqref{diagsym}, \eqref{rkna} and \eqref{rknb}
by RKN-IIIB, RKN-Diagsymp, RKN-A and RKN-B methods respectively.
These methods are applied to the following perturbed pendulum
equation
\begin{align}\label{eq:second1}
q''=-\sin q-\frac{2}{5} \cos(2q),\;q(t_0)=0,\,q'(t_0)=2.5,
\end{align}
where the initial values are taken the same as that given in
\cite{Faouhp04ecw}. The system \eqref{eq:second1} is reversible with
respect to the reflection $p\leftrightarrow -p$ (here $p=q'$) and
the corresponding Hamiltonian function (energy) is given by
$H(p,q)=\frac{1}{2}p^2-\cos q +\frac{1}{5}\sin(2q)$.

\begin{figure}[htbp]
\includegraphics[height=80mm,width=\textwidth]{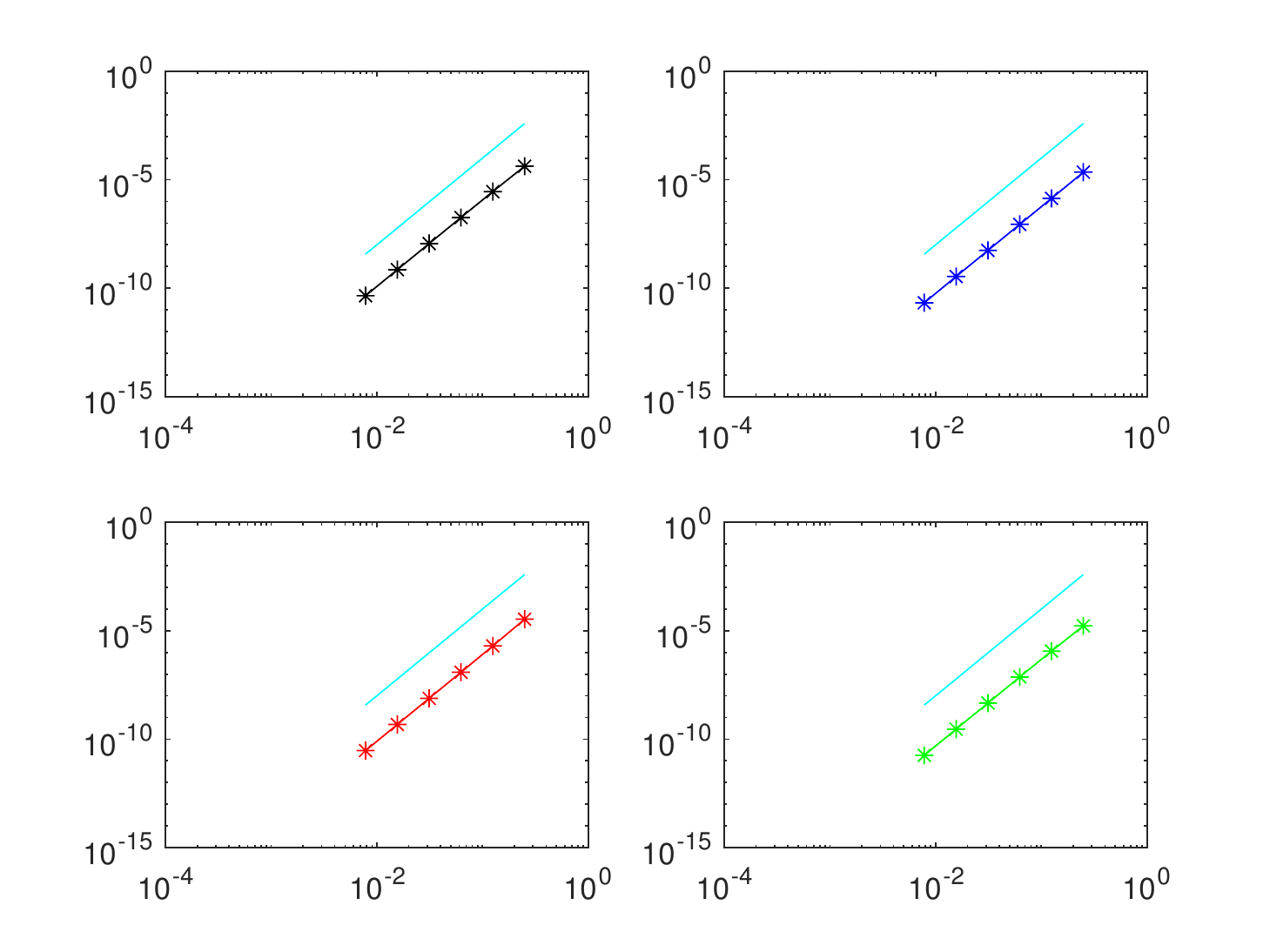}\\
\caption{Global errors of the numerical solutions by RKN-IIIB
method(black line), RKN-Diagsymp method (blue line), RKN-A method
(red line) and RKN-B method (green line) for the perturbed pendululm
equation \eqref{eq:second1}. The reference line has slope 4 in every
subplots.}\label{f0}
\end{figure}

\begin{figure}[htbp]
\includegraphics[height=80mm,width=\textwidth]{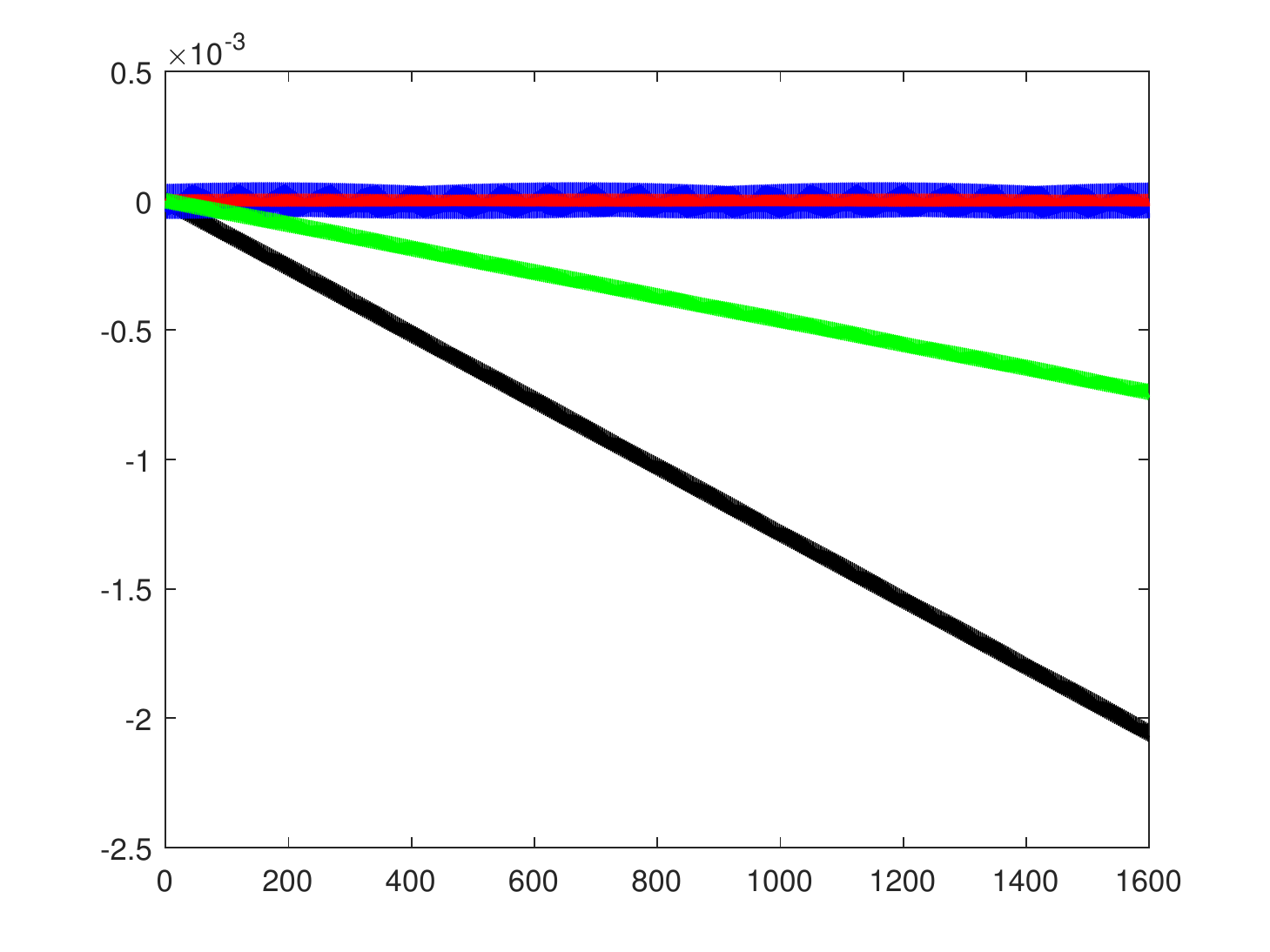}\\
\caption{Energy errors of the numerical solutions by RKN-IIIB
method(black line), RKN-Diagsymp method (blue line), RKN-A method
(red line) and RKN-B method (green line) for the perturbed pendululm
equation \eqref{eq:second1}: step size $h=0.16,$ integration
interval $[0, 1600].$}\label{f1}
\end{figure}

\begin{figure}[htbp]
\includegraphics[height=80mm,width=\textwidth]{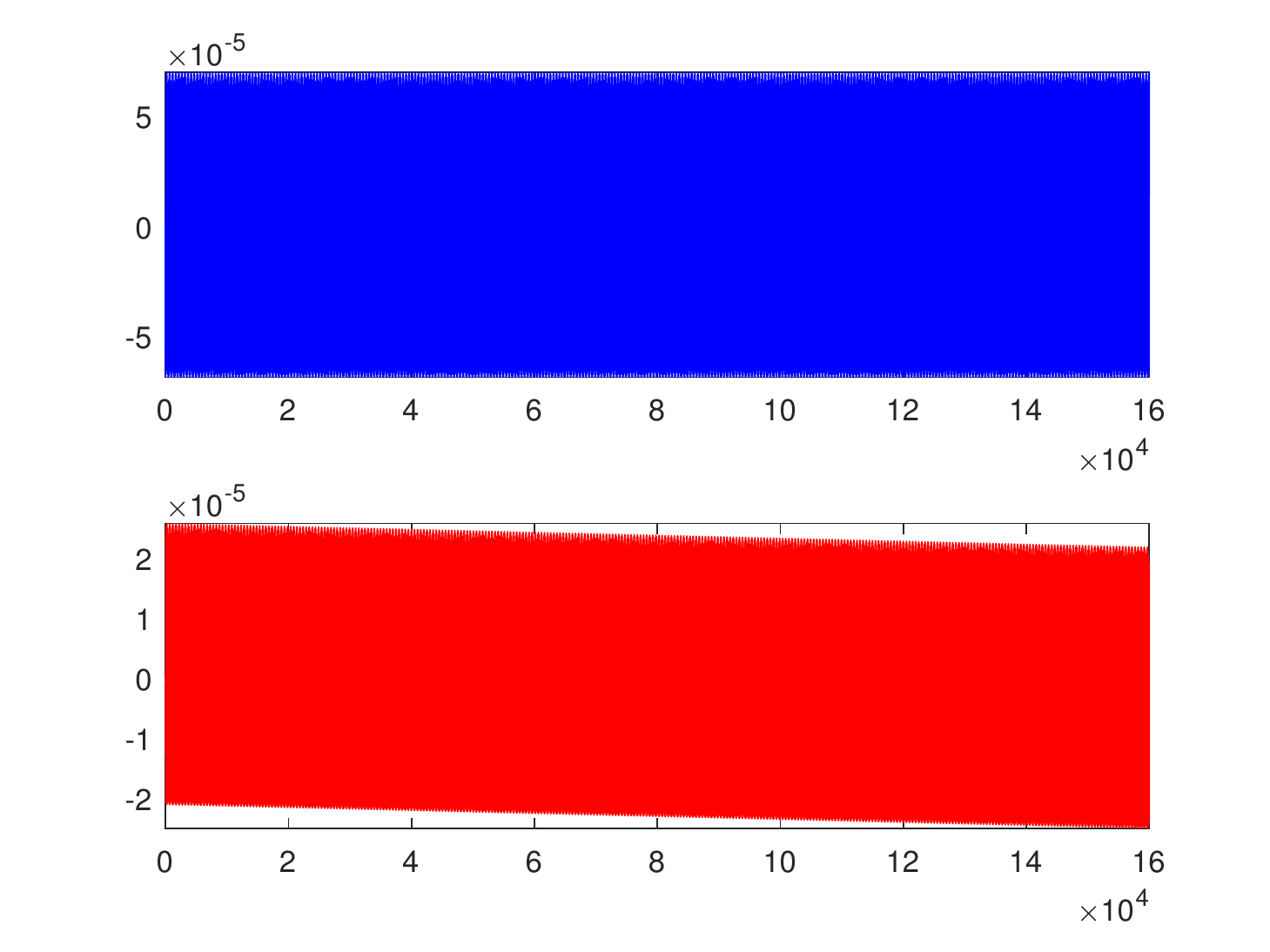}\\
\caption{Energy errors of the numerical solutions by RKN-Diagsymp
method (blue line), and RKN-A method (red line) for the perturbed
pendululm equation \eqref{eq:second1}: step size $h=0.16,$
integration interval $[0, 1.6 \times 10^5].$}\label{f2}
\end{figure}

Global errors of the numerical solutions by the above four methods
with six small step sizes are shown in Fig. \ref{f0} with log-log
scales, which verifies the order of all the methods. From Fig.
\ref{f1}, it is seen that RKN-IIIB method and RKN-B method produce
obvious energy drifts, though these methods are symmetric. This
shows that not all symmetric RKN methods nearly preserve the energy
over long times even if the system is reversible
--- this observation has been shown for symmetric Runge-Kutta methods
in \cite{Faouhp04ecw}. It is observed that the energy error keeps
bounded for the RKN-Diagsymp method. Besides, it seems that the
non-symplectic RKN-A method gives a ``better" behavior. However,
when we integrate the system on a much longer time interval $[0,
1.6\times 10^5]$, it gives a worse result (energy drift) compared
with the RKN-Diagsymp method (see Fig. \ref{f2}). From these
numerical tests we may conclude that symplectic-structure
preservation is more essential than the reversibility preservation
of the reversible Hamiltonian systems in long-term numerical
simulation. Nevertheless, for general reversible non-Hamiltonian
systems, symmetric methods are also preferable.

\section{Concluding remarks}

We develop symmetric integrators by means of continuous-stage
Runge-Kutta-Nystr\"{o}m (csRKN) methods in this paper. The crucial
technique based on Legendre polynomial expansion combining with the
symmetric conditions and order conditions is fully utilized. As
illustrative examples, new families of symmetric integrators (most
of them are also symplectic) are derived in use of Gaussian-type and
Lobatto-type quadrature formulas. It is worth observing that other
quadrature formulas can also be considered for devising symmetric
integrators and more free parameters can be led into the formalism
of the Butcher coefficients.

\section*{Acknowledgements}

The first author was supported by the National Natural Science
Foundation of China (11401055), China Scholarship Council
(No.201708430066) and Scientific Research Fund of Hunan Provincial
Education Department (15C0028). The second author was supported by
the foundation of NSFC (No. 11201125, 11761033) and PhD scientific
research foundation of East China Jiaotong University.

\end{document}